\documentclass[a4paper,11pt,reqno]{amsart}
\usepackage{amsmath}
\usepackage{amsthm,enumerate}

\usepackage{graphicx}
\usepackage{amssymb}

\usepackage{appendix}

\usepackage[italian,english]{babel}

\usepackage[utf8x]{inputenc}
\usepackage{fancyhdr}

\usepackage{calc}
\usepackage{url}

\usepackage[text={6.5in,8.6in},centering]{geometry}

\usepackage{srcltx, inputenc}

\usepackage[T1]{fontenc}

\usepackage[usenames,dvipsnames]{color}

\makeatletter
\def\cleardoublepage{\clearpage\if@twoside \ifodd\c@page\else%
         \hbox{}%
     \thispagestyle{empty}%              % Empty header styles
     \newpage%
     \if@twocolumn\hbox{}\newpage\fi\fi\fi}
\makeatother

\let\cleardoublepage\clearpage

\addto\captionsitalian{}

\newtheorem{thm}{Theorem}[section]
\newtheorem{cor}[thm]{Corollary}
\newtheorem{lem}[thm]{Lemma}

\newtheorem{den}[thm]{Definition}
\newtheorem{oss}[thm]{Remark}
\numberwithin{equation}{section}

\begin{document}

\title[PME on manifolds with critical negative curvature]{Porous medium equations on manifolds \\ with critical negative curvature: unbounded  initial data}

\author[Matteo Muratori]{Matteo Muratori} \author[Fabio Punzo]{Fabio Punzo}

% \thanks{*Corresponding author}

\address{Matteo Muratori and Fabio Punzo: Dipartimento di Matematica, Politecnico di Milano, Piaz\-za Leonardo da Vinci 32, 20133 Milano, Italy}
\email{matteo.muratori@polimi.it}
% \address{Fabio Punzo: Dipartimento di Matematica, Politecnico di Milano, Piaz\-za Leonardo da Vinci 32, 20133 Milano, Italy}
\email{fabio.punzo@polimi.it}

\maketitle

\scriptsize
\noindent{\bf Abstract.}
We investigate existence and uniqueness of solutions of the Cauchy problem for the porous medium equation on a class of Cartan-Hadamard manifolds. We suppose that the radial Ricci curvature, which is everywhere nonpositive as well as sectional curvatures, can diverge negatively at infinity with an at most \emph{quadratic} rate: in this sense it is referred to as \emph{critical}. The main novelty with respect to previous results is that, under such hypotheses, we are able to deal with \emph{unbounded} initial data and solutions. Moreover, by requiring a matching bound from above on sectional curvatures, we can also prove a \emph{blow-up} theorem in a suitable weighted space, for initial data that grow sufficiently fast at infinity. 

\normalsize

\section{Introduction}
We are concerned with existence and uniqueness of solutions of the following Cauchy problem for the \emph{porous medium equation} (PME for short) on suitable Riemannian manifolds:
\begin{equation}\label{e64}
\begin{cases}
u_t = \Delta\!\left(u^m\right) & \text{in } M \times (0,T) \, , \\
u = u_0 & \text{on } M \times \{0\} \, ,
\end{cases}
\end{equation}
where the initial datum $u_0$ is allowed to be \emph{unbounded}, $M$ is an $N$-dimensional complete, simply connected Riemannian manifold with everywhere nonpositive sectional curvatures (i.e.~a \emph{Cartan-Hadamard manifold}) and $\Delta$ is the Laplace-Beltrami operator on $M$. The power $ m $ is assumed to be strictly larger than $1$, and when dealing with sign-changing solutions we mean $ u^m := |u|^{m-1}u $. The initial data we consider that can have a specific growth at infinity, which will be discussed in detail below, and therefore one expects solutions, in general, to exist up to a certain finite time $ T>0 $ at which some \emph{blow-up} phenomena may occur. An analogous problem had originally been addressed in the celebrated paper \cite{BCP} in the case $M \equiv \mathbb R^N$, by means of techniques that seem not to adapt to the present context. Recently, quasilinear degenerate (or singular) parabolic equations of the type of \eqref{e64} on Riemannian manifolds have attracted a lot of interest: with no claim to completeness, we quote e.g.~\cite{Zhang,Pu0, Pu1, Pu02, Pu2,GMhyp,MPT,VazH, GM2,GMV,MMP,GMPlarge,GMPrm} (see also the significant papers \cite{M1,I,IM,I2,M2} in the purely parabolic case). In particular, in \cite{GMPlarge} problem \eqref{e64} has thoroughly been studied under the assumption
\begin{equation}\label{Ric-I}
\mathrm{Ric}_o(x) \geq - C_o \left(1+\mathrm{d}(x,o)^{\gamma}\right) \qquad \forall x \in M 
\end{equation}
for some $ \gamma \in (-\infty,2]$ and $ C_o>0$, where $\mathrm{Ric}_o$ stands for the \emph{Ricci curvature} evaluated in the radial direction associated with a \emph{fixed} point $o\in M$ (the \emph{pole}) and $ \mathrm{d}(\cdot,\cdot) $ is the distance on $M$. Let
\begin{equation*}\label{m3}
\sigma := \frac{2-\gamma}{2} \wedge 2 \qquad \text{and} \qquad \rho(x):= \mathrm{d}(x,o) \qquad \forall x \in M \, .
\end{equation*}
In \cite[Theorem 2.2]{GMPlarge} it is shown that if the initial datum complies with the growth estimate
\[
\left| u_0(x) \right| \leq C \left(1+\rho(x)\right)^{\frac{\sigma}{m-1}} \qquad \text{for a.e. } x\in M
\]
for some $ C>0 $, then problem \eqref{e64} admits a solution that exists at least up to a certain $T>0$ (which can be estimated depending on the behavior of $ u_0(x) $ at infinity) and satisfies
\begin{equation}\label{aa40}
|u(x,t)| \leq C_\varepsilon \left(1+ \rho(x)\right)^{\frac{\sigma}{m-1}} \qquad \text{for a.e. } (x,t) \in M \times (0,T-\varepsilon) \, ,
\end{equation}
where $ \varepsilon>0 $ can be taken arbitrarily small and $ C_\varepsilon>0 $ is locally bounded and blows up as $ \varepsilon \downarrow 0 $. Moreover, Theorem 2.3 of \cite{GMPlarge} asserts that such a solution is \emph{unique} in the class of solutions fulfilling \eqref{aa40}. Finally, in Theorems 2.5 and 2.7 of the same paper a further analysis of the \emph{maximal existence time} has been carried out: basically, it is shown that under a bound from \emph{above} on \emph{sectional curvatures} analogous to \eqref{Ric-I}, the estimate on $T$ is sharp, so there are solutions that do cease to exist in finite time.

Note that if $\gamma=2$, namely $\sigma=0$, from the aforementioned results of \cite{GMPlarge} it follows that problem \eqref{e64} is well posed, for globally bounded initial data, in $L^\infty(M \times (0, T))$ (for any $ T>0 $ actually). Accordingly, the choice $\gamma=2$ can be referred to as \emph{critical}, in the sense that when $\gamma>2$ nonuniqueness of bounded solutions prevails: see e.g.~\cite[Remark 2.4]{GMPlarge} and \cite[Remark 3.12 and Theorem 3.8]{PuAA}. The same dichotomy occurs in the linear case ($m=1$). Indeed, if $\gamma\leq 2$ the associated manifold $M$ is \emph{stochastically complete}, which is equivalent to uniqueness of bounded solutions of the Cauchy problem for the \emph{heat equation}; on the other hand, if $\gamma>2$, $M$ is in general \emph{stochastically incomplete}.

\smallskip
The goal of the present paper is to investigate existence and uniqueness of \emph{unbounded} (very weak) solutions of problem \eqref{e64} in the critical case $\gamma=2$. More precisely, throughout we assume that:
\begin{equation} \tag{{\sf H}} \label{H}
\begin{cases}
\text{(i)} & M \ \text{is a Cartan-Hadamard manifold of dimension } N \ge 2 ; \\
\text{(ii)} & \mathrm{Ric}_o(x) \geq -C_o \left(1+ \mathrm{d}(x,o)^2 \right) \ \text{for some } C_o>0, \ \text{for every } x \in M.
\end{cases}
\end{equation}
We shall prove (Theorem \ref{thmexi}) that if $ u_0 $ complies with the \emph{logarithmic-type} growth estimate 
\begin{equation*}\label{e20w}
\left|u_0(x)\right| \leq C \left[ \log\!\left (2 + \rho(x) \right) \right]^{\frac{1}{m-1}} \qquad \text{for a.e. } x \in M \, ,
\end{equation*}
for some $ C>0 $, then there exists a solution $u$ of \eqref{e64} satisfying
\begin{equation}\label{e20z}
\left|u(x,t)\right| \leq C_\varepsilon \left[\log\left(2+ \rho(x)\right)\right]^{\frac{1}{m-1}} \qquad \text{for a.e. } (x,t) \in M \times (0, T-\varepsilon)\,,
\end{equation}
where $ T > 0 $ is a suitable existence time that depends on $ u_0 $, while $ C_\varepsilon>0 $ is a locally bounded constant blowing up as $ \varepsilon \downarrow 0 $. As concerns uniqueness within the class of solutions fulfilling \eqref{e20z}, we are able to establish it (Theorem \ref{thmuniq}) under the following slightly stronger requirement: 
\begin{equation}\label{aa1}
\textrm{Ric}_o(x) \geq - C_o \, \frac{1+ \mathrm{d}(x,o)^2}{\left[\log(2+\mathrm{d}(x,o))\right]^2} \qquad \forall x \in M \, , \quad \text{for some } C_o > 0 \, .
\end{equation}
To prove the above results, we exploit basically the same line of arguments of \cite{GMPlarge}; however, we improve some of the techniques used there, in order to treat unbounded solutions. The issue of maximal existence time here is more delicate, since in \cite{GMPlarge} it deeply relied on a uniqueness result that is not available in the present context. Nevertheless, by different methods, under the additional assumption that the analogue of \eqref{H}-(ii) holds for sectional curvatures, with reverse inequality, we show that there exist solutions for which \eqref{e20z} does fail to hold after a finite time, giving rise to a sort of ``blow-up'' phenomenon at spatial infinity: see Theorem \ref{opt1} and Remark \ref{ptwse-blow} for a more complete discussion. 

\smallskip
Problem \eqref{e64} with $ M \equiv \mathbb R^N$ naturally arises in various fields of science, such as physics and biology: for a quite complete reference, see e.g.~\cite[Chapter 2]{Vaz07}. Originally, it appears as a model to describe the flow of a polytropic gas which diffuses through a \emph{porous medium}, whence the name of the equation. In fact the PME itself is obtained by combining the continuity equation of the gas (perturbed by the \emph{porosity} of the medium), \emph{Darcy's law} (an empirical relation between velocity and pressure) and the state equation for \emph{perfect gases}. In this framework, $u$ is the \emph{density} of the gas and $ p=u^{m-1} $ (up to multiplicative constants) is the \emph{pressure}. It is worth mentioning \cite{Mus}, one of the very first engineering treatises on the flow of fluids through porous media, which led to the PME. From the physical point of view, an appreciable feature of the PME is that compactly supported data expand with \emph{finite speed}: for details see again \cite[Introduction]{Vaz07}. The problem we address here, to some extent, could be associated with a PME flow (e.g.~arising from a fluid) taking place on a complete noncompact \emph{surface} $\mathcal T \subset \mathbb R^3$. In this case one obtains \eqref{e64} with $ M \equiv \mathcal T$, for which the metric, and so the Laplace-Beltrami operator, is the one induced by the embedding of $\mathcal T$ in $\mathbb R^3$. Our main results ensure, in particular, the well-posedness of problems of this type in the case of a pressure that is allowed to \emph{grow} at infinity with a logarithmic rate. 

\section{Preliminaries, assumptions and statements of the main results}\label{Stat}

In the following, we shall briefly introduce and describe all the mathematical tools necessary to our discussion, both at the level of geometry and analysis. We shall then state our results as concerns existence, uniqueness and blow-up for solutions of \eqref{e64}.

\subsection{Geometric framework} \label{RG}
Let $M$ be a complete, noncompact Riemannian manifold of dimension $ N \ge 2 $. Let $\Delta$
denote the Laplace-Beltrami operator, $\nabla$ the Riemannian
gradient and $d\mathcal{V}$ the Riemannian volume element on $M$.
Here we shall restrict ourselves to \emph{Cartan-Hadamard} manifolds, i.e.~complete, noncompact and simply connected Riemannian manifolds with everywhere \emph{nonpositive} sectional curvatures. We recall that on Cartan-Hadamard manifolds the \emph{cut locus} of any point $o\in M $ is empty and the exponential map is a diffeomorphism between $ \mathbb{R}^N $ and $ M $ (see e.g.~\cite{Grig,Grig3}). Hence, given any $x\in M\setminus \{o\}$, one can well define its \emph{polar coordinates} with respect to the pole $o$, namely $\rho(x) := \mathrm{d}(x, o)$ and $\theta\in \mathbb S^{N-1}$. Accordingly, we denote by $B_R$ the Riemannian ball of radius $R$ centered at $o$ and set $ S_R:=\partial B_R $. Through such polar coordinates, the Laplace-Beltrami operator reads
\begin{equation}\label{e1}
\Delta = \frac{\partial^2}{\partial \rho^2} + \mathsf{m}(\rho, \theta) \, \frac{\partial}{\partial \rho} + \Delta_{S_{\rho}} \, ,
\end{equation}
where $ \mathsf{m}(\rho, \theta)$ is a suitable function related to the metric tensor, which equals the Laplacian of the distance function $ x \mapsto \rho(x)  $, and $ \Delta_{S_{\rho}} $ is the Laplace-Beltrami operator on the submanifold $ S_{\rho} $.

Let
$$
\mathcal A := \left\{ \psi \in C^\infty((0,\infty)) \cap C^1([0,\infty)): \ \psi^\prime(0)=1 \, , \ \psi(0)=0 \, , \ \psi>0 \quad \text{in} \ (0,\infty) \right\} .
$$
We say that $M$ is a \emph{model manifold} (or \emph{model} for short) if the Riemannian metric $ ds $ at every $ x \equiv (\rho,\theta) $ is given by
\begin{equation*}\label{e2}
ds^2 = d\rho^2 + \psi(\rho)^2 \, d\theta^2 \, ,
\end{equation*}
where $d\theta^2$ stands for the standard metric on the unit sphere $\mathbb S^{N-1}$ and $\psi\in \mathcal A$. In this case, we write concisely $M\equiv M_\psi$. The Laplace-Beltrami operator on $ M_\psi $ has a simple form, which we shall often exploit below:
\begin{equation*}\label{mm43}
\Delta = \frac{\partial^2}{\partial \rho^2} + (N-1) \, \frac{\psi'(\rho)}{\psi(\rho)} \, \frac{\partial}{\partial\rho} + \frac1{\psi(\rho)^2} \, \Delta_{\mathbb{S}^{N-1}} \, .
\end{equation*}

Now we need to deal with curvatures. For any $x\in M\setminus\{o\}$, we denote by $\mathrm{Ric}_o(x)$ the \emph{Ricci curvature} at $x$ in the radial direction $\frac{\partial}{\partial\rho}$. Let $\omega$ be any pair of vectors from the tangent subspace $T_x M$ of the form
$$ \omega = \left(\frac{\partial}{\partial \rho} , V \right) , \qquad \text{where $V$ is any unit vector orthogonal to $\frac{\partial}{\partial\rho}$\,.} $$
We denote by $\mathrm{K}_{\omega}(x)$ the \emph{sectional curvature} at $x\in M$ of the $2$-section determined by $\omega$.
On a model manifold $M_\psi$, for every $x\equiv(\rho, \theta)\in M_\psi\setminus\{o\}$ such curvatures have the following explicit expressions:
\begin{equation}\label{e1ce}
\mathrm{K}_{\omega}(x) =-\frac{\psi^{\prime\prime}(\rho)}{\psi(\rho)} \, , \qquad \mathrm{Ric}_{o}(x) = - (N-1) \, \frac{\psi^{\prime\prime}(\rho)}{\psi(\rho)} \, .
\end{equation}
In particular, for $ M_\psi $ to be a Cartan-Hadamard manifold it is necessary to ask that $ \psi $ is convex, namely $ \psi^{\prime\prime} \geq 0 $ in $(0, \infty)$ (actually the condition is also sufficient).

Thanks to \eqref{e1ce}, straightforward computations show that \eqref{H} is satisfied e.g.~by Riemannian models $ M_\psi $ associated with convex functions $\psi$ such that
\begin{equation}\label{e50z}
\psi(\rho) = \exp\left\{ f(\rho) \right\} , \qquad f(\rho) \sim C \, \rho^{2}  \quad \text{as } \rho \to \infty \, ,
\end{equation}
where $C>0$ and $ f(\rho) \sim g(\rho) $ stands for $ \lim_{\rho \to \infty} f(\rho)/g(\rho) = 1 $. Of course Cartan-Hadamard manifolds complying with the curvature assumption \eqref{H}-(ii) can be much more general than models. Nevertheless, there are crucial Laplacian-comparison results (we refer the reader to \cite[Section 15]{Grig} or \cite[Section 2.2]{GMPlarge} for a short review) that allow us to \emph{compare} the function $ \mathsf{m}(\rho,\theta) $ on $ M $ with the corresponding function, which reads $ (N-1)\psi^\prime(\rho)/\psi(\rho) $, of a model manifold satisfying the same curvature bounds as equalities. We shall not expand this point further and only illustrate the main consequences under our specific hypotheses. First of all, by comparison with Euclidean space, on any Cartan-Hadamard manifold there holds
\begin{equation}\label{e5}
\mathsf{m}(\rho,\theta) \ge \frac{N-1}{\rho} \qquad \forall x \equiv (\rho,\theta) \in M \setminus \{ o \} \, .
\end{equation}
Moreover, by comparison with models $ M_\psi $ associated with functions $ \psi \in \mathcal{A} $ as in \eqref{e50z}, we have the following results (see e.g.~\cite[Lemma 3.3 and Lemma 5.1]{GMPlarge}).
\begin{lem}\label{prop-comp-ricci}
Let assumption \eqref{H} hold. Then there exists a positive constant $ C^\prime $, depending on $ C_o$ and $N$, such that
\begin{equation}\label{eq: m-est-above}
\mathsf{m}(\rho,\theta) \le C^\prime\,\frac{1+\rho^2}{\rho} \qquad \forall x \equiv (\rho,\theta) \in M \setminus \{ o \} \, .
\end{equation}
\end{lem}
\begin{lem}\label{prop-comp-sect}
Let \eqref{H}-\textnormal{(i)} hold. Assume in addition that
\begin{equation*}\label{assumption-sectional-prop}
\operatorname{K}_{\omega}(x) \leq - K_{o} \, \mathrm{d}(x,o)^2 \qquad \forall x \in M \setminus B_{R_o}
\end{equation*}
for some $ K_{o} , R_o > 0 $. Then there exists a positive constant $ C^{\prime\prime} $, depending on $K_{o},R_o,N$, such that
\begin{equation}\label{mm5}
\mathsf{m}(\rho,\theta) \ge C^{\prime\prime}\,\frac{1+\rho^2}{\rho} \qquad  \forall x \equiv (\rho,\theta) \in M \setminus \{ o \} \, .
\end{equation}
\end{lem}

\subsection{Functional setting}\label{sf}
In the sequel we shall consistently employ the functional space $ X_{\mathrm{log}} $, which is defined as the set of all functions $ f \in L^\infty_{\mathrm{loc}}(M) $ such that
\begin{equation*}\label{e40z}
\left|f(x)\right| \leq C \left( \log \rho(x) \right)^{\frac{1}{m-1}} \qquad \text{for a.e. } x \in M \setminus B_2
\end{equation*}
for some $ C>0 $ (depending on $f$). For each $ r \ge 2 $, we endow $ X_{\mathrm{log}} $ with the norm
\begin{equation}\label{e40-norm}
\left\| f \right\|_{\mathrm{log},r} := \sup_{x \in M} \, \frac{\left| f(x) \right|}{\left[ \log \! \left( r^2+ \rho(x)^2 \right) \right]^{\frac{1}{m-1}}} \, .
\end{equation}
Note that, by definition,
\begin{equation}\label{e40-bis}
\left|f(x)\right| \leq \left\| f \right\|_{\mathrm{log},r} \left[ \log\! \left( r^2+ \rho(x)^2 \right) \right]^{\frac{1}{m-1}} \qquad \text{for a.e. } x \in M \, .
\end{equation}
It is immediate to check that $ r \mapsto \| f \|_{\mathrm{log},r} $ is nonincreasing and
\begin{equation}\label{e40-limsup-final}
\lim_{r \to \infty}  \left\| f \right\|_{\mathrm{log},r} = \limsup_{\rho(x) \to \infty} \, \frac{\left| f(x) \right|}{\left( \log \rho(x) \right)^{\frac{1}{m-1}}} \, .
\end{equation}
We point out that, above and below, for the sake of better readability we give up the use of ``essential'' limit (or essential $\limsup$, $ \liminf $, $\sup $ and $ \inf $), even though it would be more appropriate.

\subsection{Statements of the results}\label{sect: exuni}
Before stating our main theorems, let us provide the precise notion of solution (as well as sub- and supersolution) of \eqref{e64} we deal with.

\begin{den}\label{defsol}
Let $T>0$ and $ u_0 \in L^\infty_{\mathrm{loc}}(M) $. We say that $u\in L^\infty_{\mathrm{loc}}(M \times [0, T))$ is a (very weak) solution of the Cauchy problem \eqref{e64} in the time interval $[0,T)$ if there holds
\begin{equation}\label{q50}
- \int_0^T \int_M  u \, \phi_t \,  d\mathcal{V} dt = \int_0^T \int_M  u^m \, \Delta \phi \, d\mathcal{V} dt + \int_M u_0(x) \, \phi(x,0)\, d\mathcal{V}(x)
\end{equation}
for all $\phi\in C^\infty_c(M\times [0, T))$.

Similarly, we say that $u \in L^\infty_{\mathrm{loc}}(M\times [0, T))$ is a (very weak) subsolution [supersolution] of \eqref{e64} if \eqref{q50} holds with ``$=$'' replaced by ``$\leq$'' [``$\geq$''], for all nonnegative $\phi\in C^\infty_c(M\times [0, T))$.
\end{den}

From here on we shall tacitly understand solutions, sub- and supersolutions in the \emph{very weak} sense described above, without further specifications except in cases where ambiguity might occur.

\begin{thm}[Existence]\label{thmexi}
Let assumption \eqref{H} hold and $u_0\in X_{\mathrm{log}}$. Then there exists a solution $u$ of problem \eqref{e64} in the time interval $ [0,T) $ with
\begin{equation}\label{n2}
T= \underline{C} \, \left\| u_0 \right\|^{1-m}_{\mathrm{log},r} ,
\end{equation}
where $\underline{C}$ is a positive constant depending on $C_o, N, m$, independent of $  r \ge 2 $. Moreover, such a solution satisfies the pointwise estimate
\begin{equation}\label{thm-limit-abs}
\left| u(x,t) \right| \le \left( 1 - \frac{t}{T} \right)^{-\frac{1}{m-1}} \left\| u_0 \right\|_{\mathrm{log},r} \left[\log\!\left( r^2+ \rho(x)^2 \right)\right]^{\frac{1}{m-1}} \qquad \text{for a.e. } (x,t) \in M \times (0,T) \, .
\end{equation}
In particular,
\begin{equation}\label{thm-norms}
\left\| u(t) \right\|_{\mathrm{log},r} \le \left( 1 - \frac{t}{T} \right)^{-\frac{1}{m-1}} \left\| u_0 \right\|_{\mathrm{log},r} \qquad \text{for a.e. } t \in (0,T)
\end{equation}
and
\begin{equation}\label{thm-limsup}
\limsup_{\rho(x)\to\infty} \, \frac{\left| u(x,t) \right|}{\left(\log \rho(x)\right)^{\frac 1{m-1}}} \le \left( 1 - \frac{t}{T} \right)^{-\frac{1}{m-1}} \, \limsup_{\rho(x)\to\infty} \, \frac{\left| u_0(x) \right|}{\left(\log \rho(x)\right)^{\frac 1{m-1}}} \qquad \text{for a.e. } t \in (0,T) \, .
\end{equation}
As a consequence, the solution $u$ exists in $ X_{\mathrm{log}} $ at least up to
\begin{equation}\label{thm-tmax}
T = \underline{C} \left[ \limsup_{\rho(x)\to\infty} \, \frac{\left| u_0(x) \right|}{\left(\log \rho(x)\right)^{\frac 1{m-1}}} \right]^{1-m} ,
\end{equation}
so that
\[
\lim_{\rho(x)\to\infty} \, \frac{\left| u_0(x) \right|}{\left(\log \rho(x)\right)^{\frac 1{m-1}}} = 0
\]
ensures global existence.
\end{thm}

Theorem \ref{thmexi} does not provide the exact \emph{maximal} time for which solutions of \eqref{e64} exist in $ X_{\mathrm{log}} $, but only the \emph{lower estimate} \eqref{thm-tmax}. We can show that, under an additional upper bound on sectional curvatures matching the assumed lower bound \eqref{H}-\textnormal{(ii)} on the Ricci curvature, there exist solutions associated with initial data having the prescribed logarithmic-type growth at infinity which lie in $X_{\mathrm{log}}$ \emph{at most} up to a certain time complying with an \emph{upper estimate} that mimics \eqref{thm-tmax}. 

\begin{thm}[Blow-up in $ X_{\mathrm{log}} $]\label{opt1}
Let assumption \eqref{H} hold. Suppose in addition that there exist $ K_o , R_o > 0 $ such that
\begin{equation}\label{assumption-sectional-intro}
\operatorname{K}_{\omega}(x) \leq - K_{o} \, \mathrm{d}(x,o)^2 \qquad \forall x \in M \setminus B_{R_o} \, .
\end{equation}
Let $ u_0 \in X_{\mathrm{log}} $ be any initial datum satisfying
\begin{equation}\label{tmax-opt}
\liminf_{\rho(x)\to\infty} \, \frac{u_0(x)}{\left(\log \rho(x)\right)^{\frac{1}{m-1}}} > 0 \, .
\end{equation}
Then there exists a solution $u$ of \eqref{e64} belonging to $ L^\infty_{\mathrm{loc}}\left( [0,T); X_{\mathrm{log}} \right) $, where $T$ satisfies
\begin{equation*}\label{eq:time-tau}
\underline{C} \left[ \limsup_{\rho(x)\to\infty} \, \frac{ u_0(x) }{\left(\log \rho(x)\right)^{\frac 1{m-1}}} \right]^{1-m} \le T \le \overline{C} \left[ \liminf_{\rho(x)\to\infty} \, \frac{u_0(x)}{\left( \log \rho(x) \right)^{\frac{1}{m-1}}} \right]^{1-m} 
\end{equation*}
for some $ \underline{C}=\underline{C}(C_o,N,m) >0 $ and $ \overline{C}=\overline{C}(K_o, R_o,N,m)>0 $, which blows up in $ X_{\mathrm{log}} $ at $t=T$:
\begin{equation}\label{eq:final-est-state}
\limsup_{t \uparrow T} \, \limsup_{\rho(x)\to\infty} \, \frac{u(x,t)}{\left(\log \rho(x)\right)^{\frac{1}{m-1}}} = + \infty \, .
\end{equation} 
\end{thm}

Finally, as concerns uniqueness of solutions in $ X_{\mathrm{log}} $, we have the following. 
\begin{thm}[Uniqueness]\label{thmuniq}
Let assumptions \eqref{H}-\textnormal{(i)} and \eqref{aa1} hold. Let $u, v$ be any two solutions of problem \eqref{e64} corresponding to the same initial datum $ u_0 \in X_{\mathrm{log}} $, up to the same time $T>0$, in the sense that they both satisfy \eqref{q50}. In addition, suppose that both $u$ and $v$ comply with the bound \eqref{e20z}. Then $ u =  v$ a.e.~in $M \times (0, T)$.
\end{thm}

\begin{oss}\label{rem: quad}\rm
As mentioned in the Introduction, in \cite{GMPlarge} the authors established that, under assumption \eqref{H}, there holds uniqueness of \emph{bounded} solutions of \eqref{e64}. Here we prove that, upon requiring a bit more (namely \eqref{aa1}), one can enlarge the class of uniqueness to suitable \emph{unbounded} solutions. If \eqref{H}-\textnormal{(ii)} is satisfied but \eqref{aa1} fails, we do not know whether uniqueness still holds. On the other hand, if assumption \eqref{H}-\textnormal{(ii)} is not met, in general one can construct multiple solutions in $L^\infty(M\times (0,T))$ associated with the same $u_0\in L^\infty(M)$. For a related discussion, see also \cite[Remark 2.4]{GMPlarge}.
\end{oss}

\section{Existence: proof of Theorem \ref{thmexi}}\label{existence}

First of all, let us briefly describe our strategy for an initial datum $ u_0 \in X_{\mathrm{log}}$ with $u_0 \geq 0$, which is fairly standard. For every $R>0$, one considers the approximate problems
\begin{equation}\label{q10}
\begin{cases}
u_t = \Delta \! \left(u^m\right) & \text{in } B_R\times (0,+\infty) \,, \\
u = 0 & \text{on } \partial B_R \times (0,+\infty) \,, \\
u = u_0  & \text{on } B_R\times \{0\} \,.
% \rfloor_{B_R}
\end{cases}
\end{equation}
Existence and uniqueness of a weak solution $u_R$ of \eqref{q10} is ensured by well-established results (see e.g.~\cite[Section 5]{Vaz07} in the Euclidean setting), and the corresponding comparison principle guarantees that the map $ R \mapsto u_R $ is nondecreasing. As a consequence, the pointwise limit $u:=\lim_{R\to \infty} u_R$, which is a candidate solution of \eqref{e64}, always exists: the nontrivial issue is to prove that it is finite. To this end, following similar arguments to \cite[Section 3]{GMPlarge}, we shall provide a \emph{separable supersolution} of \eqref{e64} that blows up at some positive time. However, in contrast with the results of \cite{GMPlarge}, the latter turns out to have a logarithmic-type spatial growth which matches the admissible growth of $u_0$. When sign-changing initial data are considered, such strategy can easily be adapted (we refer again to \cite[Section 3]{GMPlarge}). 

\begin{lem}\label{supersol}
Let assumption \eqref{H} hold. Given any $T,a>0$ and $ r \ge 2 $, let
\begin{equation}\label{t1c}
\overline{u}(x,t):=\left( 1 - \frac{t}{T} \right)^{-\frac{1}{m-1}} W_{T,r}(x) \qquad \forall (x,t) \in M \times [0,T) \, ,
\end{equation}
where
\begin{equation}\label{eq: expl-subsol}
W_{T,r}(x) := \frac{a \left[ \log\! \left( r^2+ \rho(x)^2 \right) \right]^{\frac{1}{m-1}}}{T^{\frac{1}{m-1}}} \qquad \forall x \in M \, .
\end{equation}
Then there exists $ a=a(C_o,N,m)>0 $ such that
\begin{equation}\label{t1a}
\overline{u}_t \geq \Delta\!\left( \overline{u}^m \right) \qquad \text{in } M \times (0, T) \, .
\end{equation}
\end{lem}
\begin{proof}
First of all, we set
\begin{equation}\label{mm18}
W(x)\equiv W(\rho(x)):= a \left[ \log\! \left( r^2+ \rho(x)^2 \right) \right]^{\frac{1}{m-1}} \qquad \forall x\equiv (\rho, \theta)\in M \, .
\end{equation}
The computation of the radial derivatives of $ W^m $ yields
\begin{equation}\label{mm6}
\left(W^m\right)^\prime\!(\rho)=a^m \, \frac{2 m}{m-1} \, \frac{\rho}{r^2 + \rho^2} \left[ \log\! \left( r^2+ \rho^2 \right) \right]^{\frac{1}{m-1}}
\end{equation}
and
\begin{equation}\label{mm7}
\begin{aligned}
\left(W^m\right)''\!(\rho) = & \, a^m \, \frac{2 m}{m-1} \, \frac{\left[ \log\! \left( r^2+ \rho^2 \right) \right]^{\frac{1}{m-1}}}{r^2 + \rho^2} - a^m \, \frac{4 m}{m-1} \, \frac{\rho^2}{\left(r^2+\rho^2\right)^2} \left[ \log\! \left( r^2+ \rho^2 \right) \right]^{\frac{1}{m-1}} \\
& + a^m \, \frac{4 m}{(m-1)^2} \, \frac{\rho^2}{\left(r^2+\rho^2\right)^2} \left[ \log\! \left( r^2+ \rho^2 \right) \right]^{\frac{2-m}{m-1}} \\
= & \, a^m \, \frac{2 m}{m-1} \, \frac{\left[ \log\! \left( r^2+ \rho^2 \right) \right]^{\frac{1}{m-1}}}{r^2 + \rho^2} \left[ 1  - \frac{2\rho^2}{r^2+\rho^2} + \frac{2\rho^2}{(m-1)\left(r^2+\rho^2\right)\log\! \left( r^2+ \rho^2 \right) } \right] .
\end{aligned}
\end{equation}
Let us check that it is possible to pick $a>0$, depending only on $ m $ and the constant $ C^\prime $ of Lemma \ref{prop-comp-ricci}, in such a way that
\begin{equation}\label{t1}
W \geq (m-1) \, \Delta\!\left(W^m\right) \qquad \text{in } M \, .
\end{equation}
Thanks to \eqref{e1}, \eqref{mm6}--\eqref{mm7} and the radiality of $W$, we can infer that
\begin{equation}\label{t1b}
(m-1)\,\Delta\!\left(W^m\right) = 2m \, a^m \, \frac{\left[ \log\! \left( r^2+ \rho^2 \right) \right]^{\frac{1}{m-1}}}{r^2 + \rho^2} \left[\rho \, \mathsf{m}(\rho, \theta) +1  - \frac{2\rho^2}{r^2+\rho^2} \left( 1 - \frac{1}{(m-1)\log\! \left( r^2+ \rho^2 \right)} \right) \right] .
\end{equation}
In particular, by virtue of \eqref{eq: m-est-above}, identity \eqref{t1b} yields (recall that $ r \ge 2 $)
\begin{equation*}\label{t2}
(m-1) \, \Delta\!\left(W^m\right) \leq 2m \, a^m \, \frac{\left[ \log\! \left( r^2+ \rho^2 \right) \right]^{\frac{1}{m-1}}}{r^2 + \rho^2} \left[ C^\prime \left( 1 + \rho^2 \right) + \frac{m+1}{m-1} \right] \qquad \text{in } M \, ,
\end{equation*}
whence inequality \eqref{t1} is fulfilled provided
\begin{equation}\label{t4a}
2m \, a^{m-1} \left[ C^\prime \left( 1 + \rho^2 \right) + \frac{m+1}{m-1} \right] \leq r^2+\rho^2 \qquad \forall  \rho \ge 0 \, .
\end{equation}
An elementary calculation shows that \eqref{t4a} holds e.g.~upon the choice
$$ a = \left[ 2m \left( C^\prime + \frac{m+1}{m-1} \right) \right]^{-\frac{1}{m-1}} . $$
This ensures the validity of \eqref{t1} and therefore the fact that $W_{T,r} $ satisfies
\begin{equation*}\label{t5}
W_{T,r} \geq (m-1)T \, \Delta\!\left( W_{T,r}^m \right) \qquad \text{in } M \, ,
\end{equation*}
which is in turn equivalent to \eqref{t1a}.
\end{proof}

\begin{proof}[Proof of Theorem \ref{thmexi}] 
For every $R>0$, let $u_R$ be the solution of problem \eqref{q10}, which in particular complies with the very weak formulation
\begin{equation}\label{t10}
- \int_0^\infty \int_{B_R} u_R \, \phi_t \,  d\mathcal{V} dt = \int_0^\infty \int_{B_R} u_R^m  \, \Delta \phi \, d\mathcal{V} dt + \int_{B_R} u_0(x) \, \phi(x,0) \, d\mathcal{V}(x)
\end{equation}
for all $\phi\in C^\infty_c(B_R\times [0, +\infty))$. Let $ \overline{u} $ be the supersolution provided by Lemma \ref{supersol}, with the choice (we can suppose with no loss of generality that $ u_0 \not\equiv 0 $)
\begin{equation*}\label{eq-sopra-T}
T= a^{m-1} \left\| u_0 \right\|_{\mathrm{log},r}^{1-m} .
\end{equation*}
Because $ u_0 \in X_{\mathrm{log}} $, as a straightforward consequence of \eqref{e40-bis} and \eqref{eq: expl-subsol} we have that
\begin{equation}\label{eq-sopra-dato}
\left|u_0(x)\right| \le W_{T,r}(x) \qquad \text{for a.e. } x \in M \, .
\end{equation}
In view of \eqref{t1a} and \eqref{eq-sopra-dato}, we can apply the comparison principle on balls to get
\begin{equation}\label{t10-comp}
u_R \le \overline{u} = \left( 1 - \frac{t}{T} \right)^{-\frac{1}{m-1}} \left\| u_0 \right\|_{\mathrm{log},r} \left[ \log\! \left( r^2+ \rho(x)^2 \right) \right]^{\frac{1}{m-1}} \qquad \text{in } B_R \times (0,T) \, .
\end{equation}
Similarly, by noting that $ -\overline{u}^m = \left( - \overline{u} \right)^m $ satisfies
\begin{equation*}\label{t1a-minus}
\left( -\overline{u} \right)_t \leq \Delta \! \left( -\overline{u} \right)^m \qquad \text{in } M\times (0, T) \, ,
\end{equation*}
we also deduce the estimate from below
\begin{equation}\label{t10-comp-minus}
- \overline{u} = - \left( 1 - \frac{t}{T} \right)^{-\frac{1}{m-1}} \left\| u_0 \right\|_{\mathrm{log},r} \left[ \log\! \left( r^2+ \rho(x)^2 \right) \right]^{\frac{1}{m-1}} \le u_R  \qquad \text{in } B_R \times (0,T) \, .
\end{equation}
From \eqref{t10-comp} and \eqref{t10-comp-minus}, there follows
\begin{equation}\label{t10-abs}
\left| u_R(x,t) \right| \le \left( 1 - \frac{t}{T} \right)^{-\frac{1}{m-1}} \left\| u_0 \right\|_{\mathrm{log},r} \left[ \log\! \left( r^2+ \rho(x)^2 \right) \right]^{\frac{1}{m-1}} \qquad \text{for a.e. } (x,t) \in B_R \times (0,T) \, .
\end{equation}
Thanks to estimate \eqref{t10-abs}, we infer that $ \{ u_R \} $ is bounded in $ L^\infty_{\mathrm{loc}}(M\times[0,T)) $, which is enough to pass to the limit in \eqref{t10} and obtain a solution $ u $ in the sense of Definition \ref{defsol} satisfying \eqref{thm-limit-abs} (see \cite[Proof of Theorem 2.2]{GMPlarge} for some more detail). As mentioned above, when $ u_0 \ge 0 $ such passage to the limit is straightforward by monotonicity. 

In order to complete the proof, we just remark that \eqref{thm-norms} is a consequence of \eqref{thm-limit-abs} and definition \eqref{e40-norm}, while \eqref{thm-limsup} and \eqref{thm-tmax} are consequences of \eqref{thm-norms} and \eqref{n2}, respectively, upon letting $ r \to \infty $ and recalling \eqref{e40-limsup-final}. 
\end{proof}

\section{Blow-up in $ X_{\mathrm{log}} $: proof of Theorem \ref{opt1}} \label{blup}

We start by a crucial result, which is the analogue of Lemma \ref{supersol}. 

\begin{lem}\label{lem: exist-subsol}
Let assumption \eqref{H}-\textnormal{(i)} hold. Suppose in addition that \eqref{assumption-sectional-intro} is fulfilled
for some $ K_{o} , R_o > 0 $. Let $T>0$. Then there exist positive constants $ a,r $, depending only on $K_o,R_o,N,m$, such that the function $ W_{T,r} $ defined in \eqref{eq: expl-subsol} satisfies 
\begin{equation}\label{subsol-ellliptic}
W_{T,r} \le (m-1) T \, \Delta\!\left( W_{T,r}^m \right) \qquad \text{in } M \, .
\end{equation}
\end{lem}
\begin{proof}
Let $W(x)\equiv W(\rho(x))$ be as in \eqref{mm18}, with $r\geq 2$. Here we need to require that $W$ satisfies
\begin{equation}\label{aa10} 
W \leq (m-1) \, \Delta\!\left(W^m\right)  \qquad \text{in } M \, .
\end{equation}
Thanks to \eqref{t1b}, \eqref{e1} and estimate \eqref{mm5} given by Lemma \ref{prop-comp-sect}, inequality \eqref{aa10} holds if
\begin{equation}\label{aa11}
\begin{aligned}
\frac{2m \, a^{m-1}}{r^2 + \rho^2} \left[C^{\prime\prime} \left(1+\rho^2 \right) +1  - \frac{2\rho^2}{r^2+\rho^2} \left( 1 - \frac{1}{(m-1)\log\! \left( r^2+ \rho^2 \right)} \right) \right] \ge 1 \qquad \forall \rho \ge 0 \, .
\end{aligned}
\end{equation}
Now we select $ r=r(C^{\prime\prime}) $ and $ a=a(r,C^{\prime\prime},m) $ so large that 
\begin{equation}\label{eq:L1}
C^{\prime\prime} \left(1+\rho^2 \right) + 1 - \frac{2\rho^2}{r^2+\rho^2} \ge \frac{C^{\prime\prime}}{2} \left(1+\rho^2 \right) \quad \forall \rho \ge 0 \qquad \text{and} \qquad a \geq \left( \frac{r^2}{C^{\prime\prime}m} \right)^{\frac 1{m-1}} .
\end{equation}
Upon the choices \eqref{eq:L1}, it is readily seen that \eqref{aa11} is fulfilled and therefore $ W $ complies with \eqref{aa10}, which is equivalent to the fact that $ W_{T,r} $ complies with \eqref{subsol-ellliptic}.
\end{proof}

\begin{proof}[Proof of Theorem \ref{opt1}]
We suppose throughout that the solutions we consider are continuous in $t$~w.r.t~the $ L^\infty_{\mathrm{loc}}(M) $ weak$^\ast$ topology: this is nonessential, but simplifies a bit the discussion.

\smallskip
Let $  0<S_n<T_n$ be suitable time sequences that will be recursively defined.~For any $ \delta>0 $, let
\begin{equation*}\label{choice-pre-pre} 
V_{n} := \left(W_{T_n,\hat{r}}^m - \delta\right)^{\frac 1 m} ,
\end{equation*}
where $ W_{T_n,\hat{r}} $ is given in \eqref{eq: expl-subsol} with $ T \equiv T_n $ and the same choices of $a$ and $r$ as in Lemma \ref{lem: exist-subsol} (depending only on $K_o,R_o,N,m$), which we label $ \hat{a} $ and $ \hat{r} $, respectively. Thanks to \eqref{subsol-ellliptic}, we have
\begin{equation}\label{choice-pre}
V_{n} \le (m-1) {T}_n \, \Delta\!\left(V_{n}^m\right) \qquad \text{in } M \, ,
\end{equation}
independently of $ \delta > 0 $. Now we set up the following iterative scheme. Let $ T_0=S_0=0 $, $ t_n:=\sum_{k=0}^n S_k $ and suppose that $ u $ is a solution of \eqref{e64} that lies in $ L^\infty\left( [0,t_n]; X_{\mathrm{log}} \right) $. We define $ T_{n+1} $ by
\begin{equation}\label{choice-T}
{T}_{n+1} = \left( \frac{\hat{a}}{1-\epsilon_{n}} \right)^{m-1} \left[ \liminf_{\rho(x)\to\infty} \, \frac{u(x,t_n)}{\left( \log \rho(x) \right)^{\frac{1}{m-1}}} \right]^{1-m} , 
\end{equation}
where $ \epsilon_0=\frac 12 $ and for $ n \ge 1 $ we select $ \epsilon_{n} \in (0,1) $ so small that
\begin{equation}\label{ch-eps}
\left[ \frac{1}{\left( 1-\epsilon_{n} \right)^{m-1}} -1 \right] \left( T_n - S_n \right) \le \frac{T_1}{2^n} \, .
\end{equation}
Note that $ T_{n+1} $ is well defined provided the $\liminf$ appearing in \eqref{choice-T} is nonzero: this is true at $n=0$ by \eqref{tmax-opt}, and will be established for all $ n $ by induction. Note that the latter is finite because $ u \in L^\infty\left( [0,t_n]; X_{\mathrm{log}} \right) $. Moreover, thanks to \eqref{choice-T} and the very definition of $W_{T_n,\hat{r}}$, there holds 
\begin{equation}\label{choice-T-2}
\lim_{\rho(x)\to\infty} \, \frac{V_{n}(x)}{\left(\log \rho(x) \right)^{\frac{1}{m-1}}} = \left(1-\epsilon_n\right)  \liminf_{\rho(x)\to\infty} \, \frac{u(x,t_n)}{\left(\log \rho(x)\right)^{\frac{1}{m-1}}} > 0 \, .
\end{equation}
Since $ u(\cdot,t_n) $ is locally bounded and \eqref{choice-T-2} holds, for $ \delta $ large enough (depending on $ n $) we can make sure that $ u(x,t_n) \ge V_n(x) $ for a.e.~$ x\in M$. Let us set 
\begin{equation*}\label{eq:subsol}
\underline{u}(x,t) := \left( 1 - \frac{t}{T_{n+1}} \right)^{-\frac{1}{m-1}} V_{n}(x) \qquad \forall (x,t) \in M \times \left[0,T_{n+1} \right)
\end{equation*} 
and observe that $ \underline{u} $ is, by construction (recall \eqref{choice-pre}), a subsolution of problem \eqref{e64} with $ u_0 \equiv u(\cdot,t_n) $ and $ T \equiv T_{n+1} $. Similarly, from Lemma \ref{supersol} we know that there exists another choice of $a$ (depending only on $ C_o,N,m $), which we label $ \tilde{a} $, such that $ \overline{u} $ defined as in \eqref{t1c} satisfies \eqref{t1a}, for all $ r \equiv \tilde{r} \ge 2 $. In addition, if we pick the  corresponding existence time $T$ as 
\begin{equation*}\label{choice-T2}
T \equiv S_{n+1}^{\tilde{r}}:= \tilde{a}^{m-1} \left[ \left\| u(\cdot,t_n) \right\|_{\mathrm{log},\tilde{r}} \right]^{1-m} \le T_{n+1}
\end{equation*}
we deduce that $ \overline{u} $ is a supersolution of problem \eqref{e64} (still with $ u_0 \equiv u(\cdot,t_n) $) which is larger than $ \underline{u} $ in the whole of $ M \times \left[0,S_{n+1}^{\tilde{r}}\right) $. Consider now the solutions $ {u}_R $ to the following nonhomogeneous Cauchy-Dirichlet problems:
\begin{equation*}\label{q10-mod} 
\begin{cases}
\left( {u}_R \right)_t = \Delta \! \left( {u}^m_R \right) & \text{in } B_R\times \left( 0,T_{n+1} \right) , \\
{u}_R = \underline{u}  & \text{on } \partial B_R \times \left( 0 , T_{n+1} \right) , \\
{u}_R = u(\cdot,t_n)  & \text{on } B_R\times \{ 0 \} \, ,
\end{cases}
\end{equation*}
for all $R>0$. By standard comparison in $B_R$, it is plain that 
\begin{equation}\label{eq:sub-sol}
{u}_R(x,t) \ge \underline{u}(x,t) = \left( 1 - \frac{t}{T_{n+1}} \right)^{-\frac{1}{m-1}} V_{n}(x)  \qquad \text{for a.e. } (x,t) \in B_R \times \left[ 0,T_{n+1}  \right) 
\end{equation}
and 
\begin{equation}\label{eq:sup-sol}
{u}_R(x,t) \le \overline{u}(x,t) = \left( 1 - \frac{t}{S_{n+1}^{\tilde{r}}} \right)^{-\frac{1}{m-1}} W_{S_{n+1}^{\tilde{r}},\tilde{r}}(x)   \qquad \text{for a.e. } (x,t) \in B_R \times \left[ 0, S_{n+1}^{\tilde{r}} \right) .
\end{equation} 
Thanks to \eqref{eq:sub-sol}--\eqref{eq:sup-sol}, if we let $ R \to \infty $ classical weak$^\ast$-convergence arguments in $ L^\infty_{\mathrm{loc}} $ guarantee that $ u_R $ converges to a solution $ u^\star $ of \eqref{e64}, with $ u_0 \equiv u(\cdot,t_n) $ and $T \equiv S_{n+1}^{\tilde{r}} $, which still complies with the bounds \eqref{eq:sub-sol}--\eqref{eq:sup-sol} in $ M \times \left[ 0,S_{n+1}^{\tilde{r}} \right) $. Since $ \tilde{r} $ can be taken arbitrarily large, by letting $ \tilde{r} \to \infty $ and recalling \eqref{e40-limsup-final} we deduce that $ u^\star $ is in $ L^\infty\left([0,S_{n+1}]; X_{\mathrm{log}} \right) $ e.g.~up to setting 
\begin{equation*}\label{def-sn}
S_{n+1} = \left(\frac{\tilde{a}}{2}\right)^{m-1} \left[ \limsup_{\rho(x)\to\infty} \, \frac{u(x,t_n)}{\left(\log \rho(x)\right)^{\frac{1}{m-1}}} \right]^{1-m} < T_{n+1} \, ,
\end{equation*}
and continues to satisfy \eqref{eq:sub-sol} in $ M \times \left[0,S_{n+1}\right] $. We have therefore proved that $u$ can be extended from $ [0,t_n] $ to $ [0,t_{n+1}]=[0,t_{n}+S_{n+1}] $ by letting $ u(\cdot,t) := u^\star(\cdot,t-t_n) $ for all $ t \in (t_n,t_{n+1}] $. Clearly such an extension lies in $ L^\infty\left([0,t_{n+1}] ; X_{\mathrm{log}} \right) $. Furthermore, by virtue of \eqref{eq:sub-sol} and \eqref{choice-T-2} there holds
\begin{equation}\label{eq:sol-tn1}
\liminf_{\rho(x)\to\infty} \, \frac{u(x,t_{n+1})}{\left( \log \rho(x) \right)^{\frac{1}{m-1}}} \ge \left( 1-\epsilon_n \right) \left( 1 - \frac{S_{n+1}}{T_{n+1}} \right)^{-\frac{1}{m-1}} \liminf_{\rho(x)\to\infty} \, \frac{u(x,t_n)}{\left(\log \rho(x)\right)^{\frac{1}{m-1}}} > 0 \, .
\end{equation}
By exploiting \eqref{eq:sol-tn1} (at $ n \equiv n-1 $), \eqref{choice-T} and \eqref{ch-eps},  for all $ n \ge 1 $ we obtain
\begin{equation}\label{choice-T-iter}
\begin{aligned}
{T}_{n+1} 
%= & \left( \frac{\hat{a}}{1-\epsilon_{n}} \right)^{m-1} \left[ \liminf_{\rho(x)\to\infty} \, \frac{u(x,t_n)}{\left( \log \rho(x) \right)^{\frac{1}{m-1}}} \right]^{1-m} \\
\le & \, \frac{1}{\left( 1-\epsilon_n \right)^{m-1}} \left( 1- \frac{S_n}{T_n} \right) \left( \frac{\hat{a}}{1-\epsilon_{n-1}} \right)^{m-1} \left[ \liminf_{\rho(x)\to\infty} \frac{u(x,t_{n-1})}{\left( \log \rho(x) \right)^{\frac{1}{m-1}}} \right]^{1-m}  \\
= & \, \frac{T_n-S_n}{\left( 1-\epsilon_n \right)^{m-1}} \le T_n-S_n + \frac{T_1}{2^n} \, .
\end{aligned}
\end{equation}
Hence, \eqref{choice-T-iter} yields
\begin{equation*}\label{eq:serie-1}
\tau := \sum_{k=0}^{\infty} S_k  \le 2 \, T_1 \quad \Longrightarrow \quad \lim_{n\to\infty} S_n = 0 \quad \Longleftrightarrow \quad \lim_{n \to \infty} \, \limsup_{\rho(x)\to\infty} \, \frac{u(x,t_n)}{\left(\log \rho(x)\right)^{\frac{1}{m-1}}}  = +\infty \, . 
\end{equation*}
We have therefore constructed a solution $ u $ of \eqref{e64} which belongs to $ L^\infty\left([0,s]; X_{\mathrm{log}} \right) $ for all $ s<\tau $ and satisfies \eqref{eq:final-est-state} with $ T=\tau $. 
\end{proof}

\begin{oss}[Blow-up in $ X_{\mathrm{log}} $ {vs}.~pointwise blow-up]\label{ptwse-blow}\rm
As mentioned in the Introduction, in \cite[Theorem 2.5]{GMPlarge} it was proved that for nonnegative initial data with critical growth there exists a time $T$ beyond which no nonnegative solution can be extended, referred to as \emph{maximal}. On the other hand, such result heavily relied on uniqueness, which here is not known if \eqref{assumption-sectional-intro} holds, since the latter is incompatible with \eqref{aa1}. In any case, Theorem \ref{opt1} establishes, without uniqueness assumptions, the exact analogue of formula (2.27) of \cite[Theorem 2.5]{GMPlarge}): we construct a solution whose $ \| \cdot \|_{\mathrm{log},r} $ norm ($ \forall r \ge 2 $) blows up at a finite time $T$. We are not able to rule out the possibility that such solution may be extended beyond $ T $, but it cannot lie in $ X_{\mathrm{log}} $ at all times.

The question whether a \emph{pointwise} blow-up also occurs has partially been answered in \cite{GMPlarge} in the case of subcritical curvatures: we refer to Theorem 2.7 there. This issue is strongly related to the existence of a positive \emph{global} solution, belonging to $ X_{\mathrm{log}} $, of the elliptic equation 
$$
(m-1) \, \Delta\!\left(\mathsf{W}^m\right)
 = \mathsf{W} \qquad \text{in } M \, ,
$$ 
which has been established in \cite[Lemma 5.5]{GMPlarge} on \emph{model manifolds}. We do not expand this point further here, but let us mention that analogues of \cite[Lemma 5.5 and Theorem 2.7]{GMPlarge} hold in the present framework as well, by means of similar computations. This ensures existence of solutions, on model manifolds having negative curvatures (according to \eqref{e1ce}) that go to infinity quadratically, exhibiting a pointwise blow-up at some time $T$ of the same type as the supersolution $ \overline{u} $ in \eqref{t1c}.
\end{oss}

\section{Uniqueness: proof of Theorem \ref{thmuniq}}\label{uniqueness}

The technique we exploit to prove uniqueness of solutions of \eqref{e64} belonging to $ X_{\mathrm{log}} $ relies on a well-established duality method, whose main ideas behind can be traced back (with no claim to completeness) to the pioneering works \cite{ACP,Pierre,BCP}. Our proof follows the lines of \cite[Proof of Theorem 2.3]{GMPlarge}: however, some nontrivial modifications are in order to handle unbounded solutions.

\begin{proof}[Proof of Theorem \ref{thmuniq}]
Throughout, $ T $ is to be understood as $ T-\varepsilon $, for arbitrary $ \varepsilon>0 $ (since solutions at the very time $ T $ may not exist). Moreover, we shall make the extra assumption that $ u , v \in C(M \times [0, T])$: if this is not the case, some technical modifications, which to the present purposes are inessential, have to be implemented (see the beginning of the proof of \cite[Theorem 2.3]{GMPlarge} for details). So, given any $R>0$ and $\xi\in C^\infty\!\left(\overline{B}_R \times [0, T]\right)$ with $ \xi=0 $ on $ S_R \times (0,T) $, since both $u$ and $v$ satisfy \eqref{q50} and are continuous, it is direct to show that 
\begin{equation}\label{n3}
\begin{aligned}
& \int_0^T \int_{B_R} \left[ (u - v) \, \xi_t + \left(u^m - v^m\right) \Delta \xi \right] d\mathcal{V} dt - \int_0^T \int_{S_R} \left(u^m-v^m\right) \frac{\partial \xi}{\partial \nu} \, d\sigma dt \\
= & \int_{B_R} \left[ u(x,T) - v(x,T) \right] \xi(x,T) \, d\mathcal{V}(x) \, ,
\end{aligned}
\end{equation}
where $\frac{\partial}{\partial \nu}$ is the outward normal derivative on $S_R$ and $ d \sigma  $ is the Hausdorff measure on $ S_R $. Let
\begin{equation}\label{n4}
a(x,t):=
\begin{cases}
\frac{u^m(x,t) - v^m(x,t)}{u(x,t)- v(x,t)} & \text{if } u(x,t)\neq v(x,t) \, , \\
0 & \text{if } u(x,t)=v(x,t) \, .
\end{cases}
\end{equation}
By virtue of \eqref{e20z} and an elementary computation, it is readily seen that 
\begin{equation}\label{n5}
0 \leq  a(x,t) \leq C_1  \, \log\!\left(2 + \rho(x)\right) \qquad \text{for a.e. } (x,t) \in M \times (0, T) \, ,
\end{equation}
for some constant $C_1>0$. After definition \eqref{n4}, we can rewrite \eqref{n3} as follows: 
\begin{equation*}%\label{n6}
\begin{aligned}
& \int_0^T \int_{B_R} (u - v) \left(\xi_t + a \, \Delta \xi \right) d\mathcal{V} dt - \int_0^T \int_{S_R} \left(u^m-v^m\right) \frac{\partial \xi}{\partial \nu} \, d\sigma dt \\
= & \int_{B_R} \left[u(x,T)- v(x,T)\right] \xi(x,T) \, d \mathcal{V}(x) \, .
\end{aligned}
\end{equation*}
Now we take $R_0 \ge  2 $, which is meant to be fixed, and $R \ge  R_0+1 \ge 3 $, which will go to infinity eventually. Let us pick a sequence of regular approximations of $ a $ such that
\begin{equation}\label{n7}
\left\{a_{n, R}\right\}_{n \in \mathbb N} \equiv  \left\{a_n\right\}_{n\in \mathbb N} \subset C^\infty(M\times [0, T]) \qquad \text{and} \qquad a_n>0 \quad \text{in}\ B_R  \quad \text{for every } n \in \mathbb N \, .
\end{equation}
Accordingly, we denote by $\xi_n$ the (regular) solution of the backward parabolic problem
\begin{equation}\label{n9}
\begin{cases}
\left( \xi_n \right)_t  + a_n \, \Delta \xi_n = 0 & \text{in } B_R \times (0, T) \, , \\
\xi_n = 0 & \text{on } \partial B_R \times (0, T) \, , \\
\xi_n = \omega & \text{on } B_R \times \{T\} \, ,
\end{cases}
\end{equation}
where $ \omega $ is an arbitrary final datum satisfying
\begin{equation}\label{n8}
\omega\in C^\infty_c(M) \, , \qquad 0 \leq \omega\leq 1 \quad \text{in } M \, , \qquad \omega =0 \quad \text{in } M \setminus B_{R_0} \, .
\end{equation}
If we plug $\xi \equiv \xi_n$ in \eqref{n3} and exploit \eqref{n9}, we obtain:
\begin{equation}\label{n14}
\begin{aligned}
& \int_0^T \int_{B_R} \left(u - v\right) \left(a- a_n\right) \Delta \xi_n \, d\mathcal{V} dt - \int_0^T \int_{S_R} \left( u^m-v^m \right)\frac{\partial \xi_n}{\partial \nu} \, d\sigma dt \\
= &  \int_{B_R} \left[ u(x,T)- v(x,T) \right]  \omega(x) \, d\mathcal{V}(x) \, .
\end{aligned}
\end{equation}
Let us label the integral quantities in the l.h.s.~of \eqref{n14}:
\begin{equation}\label{n15}
I_{n}(R) := \int_0^T \int_{B_R} \left(u-v\right) \left(a- a_n\right) \Delta \xi_n \, d\mathcal{V} dt \, ,
\end{equation}
\begin{equation}\label{n16}
J_{n}(R) := - \int_0^T \int_{S_R} \left( u^m-v^m \right)\frac{\partial \xi_n}{\partial \nu} \, d\sigma dt \, .
\end{equation}
Thanks to the growth hypothesis \eqref{e20z}, which holds both for $ u$ and $v$, there follows
\begin{equation}\label{n50aa}
\begin{aligned}
\left| J_n(R) \right|  \leq & \,  T \, \mathrm{meas}(S_R) \sup_{S_R \times (0, T)} \left|\frac{\partial \xi_n}{\partial \nu}\right|\sup_{S_R\times(0, T)} \left| u^m - v^m \right| 
\leq  C \left( \log R \right)^{\frac m{m-1}} \mathrm{meas}(S_R) \sup_{S_R\times(0, T)} \left|\frac{\partial \xi_n}{\partial \nu}\right| ,
\end{aligned}
\end{equation}
where $ \mathrm{meas}(S_R) = \sigma( S_R) $ and $ C>0 $ is a suitable constant independent of $ n $, $R$ and $R_0$. It is therefore crucial to provide an upper estimate for the normal derivative of $ \xi_n $ on $ S_R\times(0, T) $.
To this aim, we make the additional assumption that for all $R \ge R_0+1$ there exists $n_0=n_0(R)$ such that, for every $n\in \mathbb N $ with $ n \ge n_0$, the sequence $ \{ a_n \} $ satisfies
\begin{equation}\label{n20}
a_n \leq C_2 \, \log( 2+ \rho(x)  ) \qquad \forall (x,t) \in B_R \times (0, T) \, ,
\end{equation}
$C_2>0$ being another constant depending only on the constant $ C_1 $ in \eqref{n5}. This can be done with no loss of generality: one proceeds as in the corresponding part of the proof of \cite[Theorem 2.3]{GMPlarge}. 

\smallskip

{\bf Claim:} There exists $K>0$, depending only on $C_2$, such that the function 
\begin{equation}\label{n21}
\eta(x,t)\equiv \eta(\rho(x), t):= \lambda \, \exp\left\{-\frac{K}{2 T- t} \, \frac{\rho(x)^2}{\log\rho(x)}\right\} \qquad \forall(x,t)\in M \setminus B_{R_0} \times [0, T] 
\end{equation}
satisfies, for all $\lambda>0$,
\begin{equation}\label{n33}
\eta_t + a_n \, \Delta \eta \leq 0 \qquad \text{in}\ \left( B_{R} \setminus \overline{B}_{R_0} \right) \times (0, T) \, .
\end{equation}
Indeed, for all $\rho \ge R_0 $ and $ t \in (0,T)$, we have: 
\begin{equation}\label{n22}
\eta_t(\rho, t) = - \frac{ K}{(2T-t)^2} \, \frac{\rho^2}{\log \rho} \, \eta(\rho, t) < 0 \, ,
\end{equation}
\begin{equation}\label{n23}
\eta_\rho(\rho, t) = - \frac{ K }{2T-t} \, \frac{\rho \left( 2\log \rho - 1 \right)}{\left( \log\rho \right)^2} \, \eta(\rho, t) < 0
\end{equation}
and
\begin{equation*}%\label{n24}
\eta_{\rho\rho}(\rho,t) = - \frac{K \, \eta(\rho,t)}{(2T-t) \left( \log \rho \right)^4 } \left[ - \frac K{2T-t}\, \rho^2 \left(2\log \rho -1\right)^2 + 2 \left( \log \rho \right)^3 - 3 \left( \log \rho \right)^2 + 2 \log \rho \right] .
\end{equation*}
Now we observe that $R_0 \ge 2 $ entails
\begin{equation}\label{n23-bis}
2 \left( \log \rho \right)^3 - 3 \left( \log \rho \right)^2 + 2 \log \rho \ge 0 \qquad \forall \rho \ge R_0 \, .
\end{equation}
Thanks to \eqref{n22} and the fact that $ a_n \ge 0 $, in order to ensure the validity of \eqref{n33} it is not restrictive to assume that $ \Delta \eta (\rho,t) > 0 $.
Hence, upon recalling \eqref{e1}, \eqref{e5}, \eqref{n23}, \eqref{n23-bis} and \eqref{n20}, there holds
\begin{equation}\label{n26}
\begin{aligned}
\eta_t(\rho, t) + a_n\, \Delta \eta(\rho, t)
\leq & \, \eta_t(\rho, t) + C_2 \, \log\!\left(2+\rho\right) \Delta \eta(\rho,t) \\
\leq & -\frac{K}{(2T-t)^2} \, \frac{\rho^2}{\log \rho}  \, \eta(\rho,t) \left[ 1 - K \, C_2 \, \log\!\left(2+\rho\right)\frac{\left(2\log \rho -1\right)^2}{\left(\log \rho\right)^3} \right]
\end{aligned}
\end{equation}
in $ \left( B_R \setminus \overline{B}_{R_0} \right) \times (0,T) $. It is clear that one can choose $ K=K(C_2)>0 $ so small that the r.h.s.~of \eqref{n26} is negative for all $ \rho > R_0 \ge 2 $, thus the claim has been established. 

\smallskip

From \eqref{n21}, we have that 
\begin{equation}\label{n33-bis}
\lambda \geq \| \omega\|_\infty \, \exp \left\{\frac K T \, \frac{R_0^2}{\log R_0} \right\}
\qquad \Longrightarrow \qquad
\eta \geq \| \omega\|_\infty \quad \text{on } \partial B_{R_0} \times (0, T) \, .
% \left[ \left( \overline{B}_R \setminus {B}_{R_0} \right) \times \{T\} \right]
\end{equation}
Therefore, by virtue of \eqref{n9}, \eqref{n8}, \eqref{n33}, \eqref{n33-bis}, the fact that $ \eta \ge 0 $ and $ \xi_n \le \| \omega \|_\infty $, we can assert that $\eta$ is a supersolution of the following problem:
\begin{equation*}\label{n9b}
\begin{cases}
w_t  + a_n \, \Delta w = 0 & \text{in } \left(B_R \setminus \overline{B}_{R_0} \right) \times (0,T) \, , \\
w = 0 &  \text{on } \partial B_R \times (0, T) \, , \\
w = \xi_n & \text{on } \partial B_{R_0} \times (0, T) \, ,\\
w = 0 & \text{in } \left[ B_R \setminus \overline{B}_{R_0} \right] \times \{ T \} \, ,
\end{cases}
\end{equation*}
of which $ \xi_n $ is a solution. As a consequence, by arguing as in the proof of \cite[Theorem 2.3]{GMPlarge}, we end up with the estimate
\begin{equation*}\label{e2lu}
\left|\frac{\partial \xi_n(x,t)}{\partial \nu}\right| \le \frac{N-2}{R^{N-1}} \, \frac{R^{N-2}\,(R-1)^{N-2}}{R^{N-2} - (R-1)^{N-2}} \, \eta(R-1, 0) \qquad \forall (x,t) \in \partial B_R \times (0,T)
\end{equation*}
in the case $ N \ge 3 $, and with the estimate
\[
\left|\frac{\partial \xi_n(x,t)}{\partial \nu}\right| \le \frac{\eta(R-1, 0)}{R\left[\log(R)-\log(R-1)\right]} \qquad \forall (x,t) \in \partial B_R \times (0,T)
\]
in the case $ N \ge 2 $. Hence, there exists a constant $ \widehat C>0 $ (independent of $R $) such that 
\begin{equation}\label{e30aa}
\left|\frac{\partial \xi_n(x,t)}{\partial \nu}\right| \le \widehat C \, \lambda \, \exp\left\{-\frac K {2T} \, \frac{(R-1)^2}{\log(R-1)}\right\}  \qquad \forall (x,t) \in \partial B_R \times (0,T) \, .
\end{equation}
By standard surface-measure comparison results with model manifolds (see e.g.~\cite[Section 2.2]{GMPlarge}), it is direct to check that hypothesis \eqref{aa1} yields
\begin{equation}\label{e52z}
\operatorname{meas}(S_R) \leq \exp \left\{C_M \, \frac{R^2}{\log R} \right\} ,
\end{equation}
for another constant $ C_M=C_M(C_o,N)>0$. So, by combining \eqref{n50aa}, \eqref{e30aa} and \eqref{e52z}, we obtain:
\begin{equation*}%\label{n47 }
\left| J_n(R) \right| \leq C \, \widehat{C} \, \lambda \left( \log R \right)^{\frac m{m-1}} \exp\left\{C_M \, \frac{R^2}{\log R} - \frac K {2T} \, \frac{(R-1)^2}{\log(R-1)}\right\}  ,
\end{equation*}
whence
\begin{equation}\label{n52}
\limsup_{R\to \infty} \, \limsup_{n\to \infty} \left| J_n(R) \right| =  0
\end{equation}
provided $ T < K/(2 C_M) $.
Note that the constant $C_1$ of \eqref{n5}, on which $K$ depends, in principle depends on $T$; however, by its definition, it stays bounded as $ T \downarrow 0 $. As concerns the quantity $ I_n(R) $, one can reason exactly as in the final part of the proof of \cite[Theorem 2.3]{GMPlarge} and construct the sequence $ \{ a_n \} $ in such a way that \eqref{n7}, \eqref{n20} hold and
\begin{equation}\label{n53}
\limsup_{n \to \infty} \left| I_n(R) \right| = 0 \qquad \forall R \ge R_0 + 1 \, .
\end{equation}
Indeed, the technique used there is purely local (i.e.~curvature conditions at infinity do not affect it). Finally, if we let first $ n \to \infty $ and then $ R \to \infty $ in \eqref{n14}, by taking advantage of \eqref{n15}, \eqref{n16}, \eqref{n52} and \eqref{n53}, we end up with the identity
\begin{equation}\label{q11}
\int_{M} \left[ u(x,T) - v(x,T) \right] \omega(x) \, d\mathcal{V}(x) = 0 \, .
\end{equation}
Because $R_0$ can be taken arbitrarily large (the choice of $T$ is independent of $R_0$) and $ \omega $ is any regular function subject to \eqref{n8}, from \eqref{q11} we infer that $u(\cdot,T)=v(\cdot,T)$. The only constraint we still have to remove is that $T$ must be small enough: nevertheless, such restriction can be overcome by applying the above argument a finite number of times.
\end{proof}

The same technique of proof of Theorem \ref{thmuniq} yields a comparison principle.
\begin{cor}\label{comppr}
Let assumptions \eqref{H}-\textnormal{(i)} and \eqref{aa1} hold. Let $u$ and $v$ be a subsolution and a supersolution, respectively, of problem \eqref{e64} associated with the same initial datum $ u_0 \in X_{\mathrm{log}} $, up to the same time $T>0$. In addition, suppose that both $u$ and $v$ comply with the bound \eqref{e20z}. Then $ u \leq v $ a.e.~in $ M \times (0, T)$.
\end{cor}

\medskip

\noindent{\textbf{Acknowledgement.}} The authors thank the ``Gruppo Nazionale per l'Analisi Matematica, la Pro\-babilit\`a e le loro Applicazioni'' (GNAMPA) of the ``Istituto Nazionale di Alta Matematica'' (INdAM, Italy). They were both partially supported by the GNAMPA Project ``Equazioni Diffusive Non-lineari in Contesti Non-Euclidei e Disuguaglianze Funzionali Associate''.

\end{document}